\documentclass[12pt]{amsart}
\usepackage{fullpage, mathabx,url}

\newtheorem{theorem}{Theorem}[section]
\newtheorem{lemma}[theorem]{Lemma}
\newtheorem{cor}[theorem]{Corollary}

\theoremstyle{definition}
\newtheorem{defn}[theorem]{Definition}
\newtheorem{remark}[theorem]{Remark}
\newtheorem{example}[theorem]{Example}

\numberwithin{equation}{section}

\newcommand{\CC}{\mathbb{C}}
\newcommand{\FF}{\mathbb{F}}

\newcommand{\QQ}{\mathbb{Q}}

\newcommand{\ZZ}{\mathbb{Z}}

\newcommand{\frako}{\mathfrak{o}}
\newcommand{\frakq}{\mathfrak{q}}

\DeclareMathOperator{\AST}{AST}
\DeclareMathOperator{\Aut}{Aut}

\DeclareMathOperator{\End}{End}
\DeclareMathOperator{\Gal}{Gal}
\DeclareMathOperator{\GL}{GL}

\DeclareMathOperator{\image}{image}

\DeclareMathOperator{\Out}{Out}
\DeclareMathOperator{\PGL}{PGL}
\DeclareMathOperator{\rank}{rank}
\DeclareMathOperator{\Res}{Res}

\DeclareMathOperator{\Sp}{Sp}
\DeclareMathOperator{\Spec}{Spec}

\begin{document}

\title{Endomorphism fields of abelian varieties}
\author{Robert Guralnick and Kiran S. Kedlaya}
\date{May 30, 2017}
\thanks{Thanks to Francesc Fit\'e, Ga\"el R\'emond, Jean-Pierre Serre, Alice Silverberg, and Andrew Sutherland for feedback. Guralnick was partially supported by NSF grants
DMS-1302886 and DMS-1600056. Kedlaya was supported by NSF grant DMS-1501214 and UC San Diego
(Warschawski Professorship).}

\maketitle

\begin{abstract}
We give a sharp divisibility bound, in terms of $g$, for the degree of the field extension required to realize the endomorphisms of an abelian variety of dimension $g$ over an arbitrary number field; this refines a result of Silverberg. This follows from a stronger result giving the same bound for the order of the component group of the Sato-Tate group of the abelian variety, which had been proved for abelian surfaces by Fit\'e--Kedlaya--Rotger--Sutherland. The proof uses Minkowski's reduction method, but with some care required in the extremal cases when $p$ equals 2 or a Fermat prime.
\end{abstract}

\section{Introduction}

For $A$ an abelian variety over a field $K$, the \emph{endomorphism field} of $A$ is the minimal algebraic extension $L$ of $K$ such that $\End(A_L) = \End(A_{\overline{L}})$. The purpose of this paper is to establish a bound on the degree $[L:K]$ in terms of the dimension of $A$; more precisely, we compute the LCM of all possible degrees as $A,K$ vary while $\dim_K A$ remains fixed.

Before stating our result, we state a prior result of Silverberg \cite{silverberg} which already contains many of the main ideas. For $g$ a positive integer and $p$ a prime, define
\[
r(g,p) := \sum_{i=0}^\infty \left\lfloor \frac{2g}{(p-1)p^i} \right\rfloor.
\]
\begin{theorem}[Silverberg] \label{T:silverberg}
For $A$ an abelian variety of dimension $g$ over a number field $K$, the endomorphism field of $A$ is a finite Galois extension of $K$ of degree dividing $2 \times \prod_p p^{r(g,p)}$.
\end{theorem}
The proof \cite[Theorem~4.1]{silverberg} is elegantly simple: one verifies that for each prime $\ell > 2$, the Galois group of the endomorphism field extension is isomorphic 
(via its action on $\ell$-torsion points) to a subquotient of the group $\Sp(2g, \FF_\ell)$.
The bound is then obtained by taking the greatest common divisor of the orders of these finite groups. This echoes the method used by Minkowski to bound the order of a finite group of integer matrices (as exposed in \cite{guralnick}); we will return to this analogy a bit later.

From this proof, it is not clear whether one should expect the bound of Theorem~\ref{T:silverberg} to be sharp. However, a moment's thought shows that it is not sharp for $g=1$: the bound is $2^4 \times 3$ but the optimal bound is obviously 2, with the worst case being an elliptic curve whose CM field is not contained in $K$. More seriously, for $g=2$, the bound of Theorem~\ref{T:silverberg} is $2^8 \times 3^2 \times 5$, but (as will be explained shortly) the work of Fit\'e--Kedlaya--Rotger--Sutherland \cite{fkrs} implies that the optimal bound is $2^4 \times 3$. This raises the question of identifying the discrepancy between Theorem~\ref{T:silverberg} and the optimal bound, and this is achieved by our main result.

\begin{theorem} \label{T:endomorphism field}
For $A$ an abelian variety of dimension $g > 0$ over a number field $K$, the degree over $K$ of the endomorphism field of $A$ divides
\[
\prod_p p^{r'(g,p)}, \qquad
r'(g,p) := \begin{cases} r(g,p) - g - 1 & \mbox{if $p=2$;} \\
\max\{0, r(g,p)-1\} & \mbox{if $p$ is a Fermat prime;} \\
r(g,p) &\mbox{otherwise.}
\end{cases}
\]
Moreover, for fixed $g,p$ (but varying over all $K$), this value of $r'(g,p)$ is best possible.
\end{theorem}

To give one more example, for $g=3$, the bound from Theorem~\ref{T:silverberg}
is $2^{11} \times 3^4 \times 5 \times 7$ whereas the optimal bound is
$2^6 \times 3^3 \times 7$. It is easy to see that the factor of 7 is necessary, e.g., by considering twists of the Klein quartic (see \cite[\S 4]{poonen-schaefer-stoll}).

As in \cite{fkrs}, our approach uses the relationship between the endomorphism field $L$ of $A$ and the \emph{Sato-Tate group} of $A$; the latter is a compact Lie group whose component group surjects canonically onto $\Gal(L/K)$, and the bound we ultimately prove is for the order of the component group (see Theorem~\ref{T:Sato-Tate group}). The Sato-Tate group is constructed as a compact form of a certain linear algebraic group over $\QQ$, the \emph{algebraic Sato-Tate group}, which allows us to bound the order of the component group using a variant of Minkowski's method. One key point is that the extremal cases occur for CM abelian varieties, for which the connected part of the algebraic Sato-Tate group is a torus which splits over a CM field; what distinguishes a Fermat prime $p$ in this context is that
$\QQ(\mu_p)$ contains no proper subfield which is CM. (For $p=2$, the same statement about $\QQ(\mu_4)$ plays an analogous role.)

To prove that Theorem~\ref{T:endomorphism field} is sharp, we use the relationship between twisting of abelian varieties and Sato-Tate groups; this reduces the problem to exhibiting abelian varieties admitting actions by large finite groups, which we achieve using CM abelian varieties and the same wreath product construction as in Minkowski's theorem. Note that the fields of definition of the resulting abelian varieties are controlled by class groups of abelian number fields, so we are unable to establish lower bounds over any fixed number field.

To conclude this introduction, we comment on the subtler problem of giving bounds by size, rather than divisibility. As described in
\cite[\S 6]{guralnick}, results of Weisfeiler and Feit can be combined with the classification of finite simple groups to show that for $n > 10$, the largest finite subgroups of $\GL(n, \QQ)$ have order $2^n n!$ (and are unique up to conjugacy). For abelian varieties of sufficiently large dimension $g$, one would expect that the largest possible endomorphism field extension, and the largest possible component group of the Sato-Tate group,
are obtained by twisting a power of an elliptic curve with $j$-invariant 0 using an automorphism group of order $6^g g!$ (note that these examples already occur over $\QQ$). For the endomorphism field, this expectation has been confirmed by work of R\'emond \cite{remond}; it is highly likely that a similar analysis applies to the component group (because the cases where the two differ tend not to have enough CM to trouble the bounds), but this would require some additional argument.

\section{Group schemes}

We start with some notation. Our notation choices are not entirely typical; they are made to help us distinguish between groups and group schemes.
\begin{defn} \label{D:group scheme}
For $G$ a group scheme (over some base), we write $G_X$ for the base extension of $G$ to the base scheme $X$, and $G(X)$ for the group of $X$-valued points of $G$.
 We say that $G$ is \emph{pointful over $X$} if $G(X)$ occupies a Zariski-dense subset of $G$.
By convention, all group schemes we consider will be smooth and linear; the standard linear groups will be considered as schemes over $\ZZ$, and we will write $\GL(n,X)$ instead of $\GL(n)(X)$ and so on.

For $G$ a group scheme over a connected base, let $G^\circ$ denote the identity connected component, and write $\pi_0(G) := G/G^\circ$ for the group of connected components, viewed as a finite group scheme over the same base. If $G$ is pointful, then so are both $G^\circ$
and $G/G^\circ$.

For $L/K$ a finite extension of fields and $G$ a group scheme over $\Spec L$, write
$\Res^L_K G$ for the Weil restriction of scalars of $G$ to $\Spec K$.
\end{defn}

\begin{example}
For $n$ a positive integer, the $n$-torsion subscheme of the multiplicative group over $\QQ$ is the group scheme $\Spec \QQ[x]/(x^n-1)$ which is obviously defined over $\QQ$. However, it is only pointful for $n=1,2$.
\end{example}

\begin{defn}
For $G$ a group scheme, let $\Out(G)$ be the group scheme of outer automorphisms of $G$,
i.e., the cokernel of the map $G \to \Aut(G)$ induced by conjugation. Note that for $G$ a group scheme over a field $k$ which is not algebraically closed, an element of $\Aut(G)(k)$ may map trivially to $\Out(G)(k)$ even though it does not come from the image of $G(k)$; that is, the scheme-theoretic notion of an outer automorphism disagrees with the group-theoretic notion because the latter is not stable under base extension.
\end{defn}

\section{Minkowski's method}
\label{sec:Minkowski}

We next formulate our version of Minkowski's reduction method.
We implicitly follow \cite{guralnick}, but see also \cite{serre} for another detailed treatment (both considering only finite groups).

Throughout \S\ref{sec:Minkowski}, let $G$ be a group scheme over a number field $K$.
\begin{defn} \label{D:Schur-Minkowski}
By convention $G$ is a scheme of finite type over $K$, so it can be extended to a group scheme over $\frako_K[1/N]$ for some positive integer $N$. In particular, it makes sense to form the base extension of $G$ to $\FF_{\frakq} := \frako_K/\frakq$ for all but finitely many prime ideals $\frakq$ of $\frako_K$.

Let $H$ be a finite pointful subquotient group scheme of $G$. By the previous paragraph, $H (K)$ is isomorphic to a subquotient of $G(\FF_{\frakq})$ for all but finitely many
$\frakq$; in particular, for each prime $p$, any $p$-Sylow subgroup of $H(K)$ is isomorphic to a subquotient of a $p$-Sylow subgroup of $G(\FF_{\frakq})$ for all but finitely many $\frakq$.

To translate this into a numerical bound, define the nonnegative integers $r(G,p)$ by the formula
\[
\prod_p p^{r(G,p)} = \sup_S \{\gcd_{\frakq \in S} \#G(\FF_{\frakq})\}
\]
where $S$ runs over all cofinite sets of prime ideals of $\frako_K$. Then the preceding discussion implies that $H$ has order dividing $\prod_p p^{r(G,p)}$; in particular, this bound applies to the component group of any pointful subgroup scheme of $G$.
\end{defn}

We collect some remarks related to this construction.
\begin{remark} \label{R:sharp Minkowski1}
Suppose that there exists a finite pointful subquotient group scheme $H$ of $G$ such that the $p$-part of $\#H$ equals the upper bound $p^{r(G,p)}$. Let $P$ be a $p$-Sylow subgroup of $H$; then for infinitely many $\frakq$, $P$ has the same cardinality as a $p$-Sylow subgroup of $G(\FF_{\frakq})$,
so by Sylow's theorem the two must be isomorphic. That is, the isomorphism type of $P$ is uniquely determined by $G$.
\end{remark}

\begin{remark} \label{R:sharp Minkowski}
Let $H$ be a pointful subgroup scheme of $G$. Then $\#\pi_0(H)$ divides $\prod_p p^{r(G,p)-\delta(G,p)}$ for
\[
\prod_p p^{\delta(G,p)} = \inf_S \{\gcd_{\frakq \in S} \#H^\circ(\FF_{\frakq})\}.
\]
\end{remark}

\begin{remark} \label{R:twisted forms}
In light of Wedderburn's theorem, for any number field $K$ and any twisted form $G$ of $\GL(n)_K$
one has $r(G, p) = r(\GL(n)_K, p)$.
\end{remark}

\begin{remark} \label{R:Minkowski Q}
For each prime $p$, the group $C_p \wr S_{\lfloor n/(p-1) \rfloor}$ embeds into
$\GL(n, \QQ)$, as then do its $p$-Sylow subgroups. The original theorem of Minkowski
asserts that the conjugates of the latter are the largest possible $p$-subgroups of $\GL(n, \QQ)$. However, if we compare this to the values
\begin{align*}
r(\GL(n)_\QQ, p) &= \left\lfloor \frac{n}{p-1} \right\rfloor + \left\lfloor \frac{n}{p(p-1)} \right\rfloor + \left\lfloor \frac{n}{p^2(p-1)} \right\rfloor + \cdots \qquad (p > 2) \\
r(\GL(n)_\QQ, 2) &= n + 2 \left\lfloor \frac{n}{2} \right\rfloor + \left\lfloor \frac{n}{4} \right\rfloor + \cdots,
\end{align*}
we see that the Minkowski bound is only sharp for $p>2$; for $p=2$, the Minkowski bound is too large by $\left\lfloor \frac{n}{2} \right\rfloor$, and one must supplement using some extra analysis involving quadratic forms over finite fields \cite[\S 5]{guralnick}.
For this reason, we do not know whether the example $C_p \wr S_{\lfloor n/(p-1) \rfloor}$
is optimal also for subquotients when $p=2$.
\end{remark}

\begin{remark} \label{R:Minkowski K}
For $K$ a number field, the Chebotarev density theorem implies that $r(\GL(n)_K, p)$ depends only on $K \cap \QQ(\mu_{p^\infty})$. For
\begin{align*}
m(K,p) &:= \min\{m \geq 1: K \cap \QQ(\mu_{p^m}) = K \cap \QQ(\mu_{p^\infty})\} \\
t(K,p) &:= [\QQ(\mu_{p^{m(K,p)}}): K \cap \QQ(\mu_{p^{m(K,p)}})],
\end{align*}
for $p>2$ we have
\begin{equation} \label{eq:minkowski k odd}
r(\GL(n)_K, p) = m(K,p) \left\lfloor \frac{n}{t(K,p)} \right\rfloor + \left\lfloor \frac{n}{pt(K,p)} \right\rfloor + \left\lfloor \frac{n}{p^2 t(K,p)} \right\rfloor + \cdots \qquad
\end{equation}
and this bound is again optimal (see \cite[\S 5.3]{guralnick} for a detailed discussion).

For $p = 2$, the situation depends crucially on whether $\QQ(\zeta_4) \subseteq K$. If so,
then $t(K,2) = 1$,
\begin{equation} \label{eq:minkowski k 2}
r(\GL(n)_K, 2) = m(K,2) n + \left\lfloor \frac{n}{2} \right\rfloor + \left\lfloor \frac{n}{4} \right\rfloor + \cdots,
\end{equation}
and this bound is optimal, achieved by $C_{2^{m(K,2)}} \wr S_n$. If not, the situation is more complicated; we limit ourselves to observing that for $K = \QQ(\sqrt{-2})$ we have $r(\GL(n)_K, 2) = r(\GL(n)_\QQ, 2)$,
while for $K = \QQ(\sqrt{2})$ we have $r(\GL(n)_K, 2)= r(\GL(n)_\QQ, 2) + \left\lfloor \frac{n}{2} \right\rfloor$.
\end{remark}

\begin{remark}
Although we do not need this for our main result, we note in passing the following corollary of Remark~\ref{R:Minkowski K} (which only affects the prime $p=2$): the optimal divisibility bound for the order of a finite subgroup of $\Sp(2g, \QQ)$ is $2^{r(\GL(g)_{\QQ(i)}, 2)} \prod_{p>2} p^{r(\GL(2g), p)}$. This comes down to the fact that any irreducible finite subgroup of $\Sp(2g, \QQ)$ is centralized by some totally imaginary number field \cite[Lemma~2.3]{kirschmer}.
\end{remark}

\begin{remark}
It is natural to use Minkowski's method as a starting point for bounding the order of finite subgroups of any reductive group over any field. This has been discussed extensively by Serre \cite{serre}.
\end{remark}

\section{Comparison of Minkowski bounds}

For our purposes, it will be important to compare the Minkowski bounds for various group/subgroup pairs. The key points will be to identify discrepancies for $p=2$, and to isolate cases for $p>2$ where the bound for the subgroup matches that of the full group.

\begin{remark} \label{R:double}
A trivial but useful observation along these lines is that for $n \geq 1$, $p>2$, $d > 1$,
\[
r(\GL(dn)_\QQ, p) > r(\GL(n)_\QQ, p) \qquad \mbox{whenever $r(\GL(dn)_\QQ, p)>0$}.
\]
A slightly less trivial observation is that for $n \geq 1$,
\[
r(\GL(n)_{\QQ(i)}, 2) > r(\GL(n)_\QQ, 2).
\]
\end{remark}

\begin{lemma} \label{L:imprimitive1}
Let $K$ be a number field of degree $d$ over $\QQ$. For each integer $n \geq 1$
and each odd prime $p$,
\[
r(\GL(dn)_\QQ, p) \geq r(\GL(n)_K, p) + r(S_d, p);
\]
moreover, if $n>1$ and $m(K,p) >1$, or if $d \geq p(p-1)$, then the inequality is strict.
\end{lemma}
\begin{proof}
There is nothing to check when $d=1$, so we may assume $d \geq 2$.
Put $m = m(K,p)$, $t = t(K,p)$; then the desired inequality is
\[
\left\lfloor \frac{dn}{p-1} \right\rfloor +
\left\lfloor \frac{dn}{p(p-1)} \right\rfloor + \cdots
\geq 
m \left\lfloor \frac{n}{t} \right\rfloor
+ \left\lfloor \frac{n}{pt} \right\rfloor + \cdots
+ \left\lfloor \frac{d}{p} \right\rfloor
+ \left\lfloor \frac{d}{p^2} \right\rfloor
+ \cdots.
\]
From the equality $dt = p^{m-1}(p-1)$, we see that $\frac{d}{p-1}$ equals $\frac{1}{t}$ times the integer $p^{m-1}$ which is no less than $m$ (and strictly greater than $m$ if $m>1$).
Consequently, by writing the difference between the two sides as
\[
\left\lfloor \frac{dn}{p-1} \right\rfloor
- 
m \left\lfloor \frac{n}{t} \right\rfloor
+ \left\lfloor \frac{dn}{p(p-1)} \right\rfloor 
- \left\lfloor \frac{n}{pt} \right\rfloor + \cdots
+ \left\lfloor \frac{d}{p} \right\rfloor
+ \left\lfloor \frac{d}{p^2} \right\rfloor
+ \cdots,
\]
we see that this difference does not decrease if we increase $n$ by 1 (and strictly increases if $m>1$). If $n=1$, then the desired inequality becomes $r(\GL(d)_\QQ, p) \geq r(S_d, p)$, which holds because $S_d$ embeds into $\GL(d, \QQ)$; this equality is strict whenever $d \geq p(p-1)$. This proves the claim.
\end{proof}

\begin{cor} \label{C:imprimitive2}
For $K$ a number field of degree $d$, for $p>2$ we have
\begin{equation} \label{eq:imprimitive2}
r(\GL(dn)_\QQ, p) \geq r(\Aut(K/\QQ) \ltimes \Res^K_{\QQ} \GL(n)_K, p).
\end{equation}
Moreover, if $K \not\subseteq \QQ(\mu_p)$, $K$ is not the degree-$p$ subextension of $\QQ(\mu_{p^2})$, 
and $r(\GL(dn)_\QQ, p) \neq 0$, then the equality is strict.
\end{cor}
\begin{proof}
The inequality \eqref{eq:imprimitive2}
holds because $\Aut(K/\QQ) \ltimes \Res^K_{\QQ} \GL(n)_K$ embeds into
$\GL(dn)_\QQ$. We thus only need to obtain a contradiction assuming that
$K \not\subseteq \QQ(\mu_p)$, $K$ is not the degree-$p$ subextension of $\QQ(\mu_{p^2})$, $r(\GL(dn)_\QQ, p) > 0$,
and equality holds in \eqref{eq:imprimitive2}.

Note that \eqref{eq:imprimitive2} also follows from Lemma~\ref{L:imprimitive1}, so equality must also hold in the latter.
By Remark~\ref{R:Minkowski K}, if $K' := K \cap \QQ(\mu_{p^\infty})$
has degree $d' \neq d$, then $r(\GL(n)_K,p) = r(\GL(n)_{K'}, p)$; we then get the strict inequality
by applying Lemma~\ref{L:imprimitive1} to the field $K'$ and invoking Remark~\ref{R:double}
(using the condition that $r(\GL(dn)_\QQ, p) > 0$).
We must therefore have $K = K'$ and hence $K \subseteq \QQ(\mu_{p^\infty})$;
since we assumed $K \not\subseteq \QQ(\mu_p)$, this implies that $m(K,p) > 1$.
To have equality in Lemma~\ref{L:imprimitive1}, we must then have $n=1$ and $d < p(p-1)$.

Since $m(K,p) > 1$, $K$ must contain the degree-$p$ subextension of $\QQ(\mu_{p^2})$, necessarily strictly by hypothesis; hence $d/p$ is an integer strictly greater than 1.
By the bound on $d$, we cannot then have $\QQ(\mu_p) \subseteq K$,
so $r(\GL(1)_K, p) = 0$.
Meanwhile, $K$ is an abelian extension of $\QQ$, so $r(\Aut(K/\QQ), p) = 1$.
However, since $d \geq 2p$, $r(\GL(d)_\QQ, p) \geq 2$, yielding the desired contradiction.
\end{proof}

\begin{cor} \label{C:imprimitive3}
For any integers $n,d > 1$, for each odd prime $p$ for which $r(\GL(nd)_\QQ, p) > 0$, we have
\[
r(\GL(dn)_\QQ, p) > r(\GL(n)_\QQ, p) + r(S_d, p).
\]
\end{cor}
\begin{proof}
We first reduce to the case where $d$ is even.
Namely, we may do this by replacing $d$ with $2 \left\lfloor \frac{d}{2} \right\rfloor$
except if $d$ is odd, $r(\GL(dn)_\QQ, p) > 0$, and $r(\GL((d-1)n)_\QQ, p) = 0$.
This implies $(d-1) n < p-1 \leq dn$, which implies on one hand that $n < p-1$ and $r(\GL(n)_\QQ, p) = 0$, and on the other hand that $d-1 < p-1$ and so $r(S_d, p) = 0$;
this yields the claimed inequality.

For any number field $K$ of even degree $d$,
by Lemma~\ref{L:imprimitive1} we have
\[
r(\GL(dn)_\QQ, p) \geq r(\GL(n)_K, p) + r(S_d,p) \geq r(\GL(n)_\QQ, p) + r(S_d, p),
\]
so it suffices to confirm that both equalities cannot hold simultaneously.
 Since we are free to choose $K$, we take it to contain the quadratic subfield of $\QQ(\mu_p)$; then by Remark~\ref{R:Minkowski K}, we have $r(\GL(n)_K, p) > r(\GL(n)_\QQ,p)$ unless both quantities equal zero. It thus suffices to rule out the equality
$r(\GL(dn)_\QQ, p) = r(S_d, p)$; since $r(\GL(d)_\QQ, p) \geq r(S_d, p)$, this follows from Remark~\ref{R:double}.
\end{proof}

For $p=2$, we have the following analogue of Corollary~\ref{C:imprimitive2}.
\begin{remark} \label{R:base extension}
For $K/F$ an extension of number fields of degree $d$, we obviously have
\[
r(\GL(dn)_{F}, 2) \geq r(\Aut(K/F) \ltimes \Res^K_{F} \GL(n)_K, 2).
\]
In case $F = \QQ(i)$, one can show using Remark~\ref{R:sharp Minkowski1} that equality holds only when $d=1$; however, we will not need this.
\end{remark}

For $p=2$, we have the following analogue of Corollary~\ref{C:imprimitive3}.

\begin{lemma} \label{L:imprimitive}
For any integers $n,d > 1$ such that $g = dn/2$ is an integer,
\[
r(\GL(g)_{\QQ(i)}, 2) \geq r(\GL(n)_\QQ, 2) + r(S_d, 2)
\]
with equality only for $(n,d) = (2,2)$.
\end{lemma}
\begin{proof} 
We are claiming that
\[
dn + \left\lfloor \frac{dn}{4} \right\rfloor + \left\lfloor \frac{dn}{8} \right\rfloor + \cdots \geq n + 2 \left\lfloor \frac{n}{2} \right\rfloor + \left\lfloor \frac{n}{4} \right\rfloor + \cdots
+ \left\lfloor \frac{d}{2} \right\rfloor + \left\lfloor \frac{d}{4} \right\rfloor + \cdots
\]
with equality only for $(n,d) = (2,2)$. For $d=2$, this inequality becomes
$n \geq \left\lfloor \frac{n}{2} \right\rfloor + 1$; for $n=2$, it becomes
$2d \geq 4$. It thus remains to check the cases where $n, d \geq 3$.

We may write the difference between the two sides as
\[
(d-1)n + \left(\left\lfloor \frac{dn}{4} \right\rfloor - \left\lfloor \frac{d}{2} \right\rfloor \right)
+ \left(\left\lfloor \frac{dn}{8} \right\rfloor - \left\lfloor \frac{d}{4} \right\rfloor \right) + \cdots
- 2 \left\lfloor \frac{n}{2} \right\rfloor - \left\lfloor \frac{n}{4} \right\rfloor - \cdots;
\]
in particular, for $n \geq 2$ fixed, this difference increases if we increase $d$ by 2.
Using the previous paragraph, we deduce the claim when $d$ is even.
For $d$ odd, we may argue that
\[
r(\GL(g)_{\QQ(i)},2) \geq r(\GL(g - n/2)_{\QQ(i)}, 2)
\geq r(\GL(n)_\QQ, 2) + r(S_{d-1}, 2)
= r(\GL(n)_\QQ, 2) + r(S_d, 2)
\]
with strict inequality if $n > 2$.
\end{proof}

\section{Abelian varieties and Sato-Tate groups}

We now specialize to the cases of interest for abelian varieties.
\begin{defn}
For $A$ an abelian variety over a number field $K$, let $\AST(A)$ denote the \emph{algebraic Sato-Tate group} of $A$ in the sense of \cite[Definition~9.5]{bk}. The key properties that we need are the following.
\begin{itemize}
\item
The group scheme $\AST(A)$ is a pointful subgroup scheme of $\Sp(2g)_\QQ$ whose connected part is reductive. The connected part is closely related to the \emph{Mumford-Tate group} of $A$.
\item
There exists a torus $T \subset \AST(A)^\circ_\CC$ which acts on $\CC^{2g}$ with weights $1, -1$ each with multiplicity $g$. In particular, the fixed space of $\AST(A)$ is the zero subspace.
\item
The component group $\pi_0(A)$ surjects onto $\Gal(L/K)$ for $L$ the endomorphism field of $A$; this map is a bijection whenever the Mumford-Tate group is completely explained by endomorphisms (which holds in all cases when $g \leq 3$ but can fail for $g \geq 4$, as originally shown by Mumford).
Moreover, for $K'$ a finite extension of $K$, $\AST(A_{K'})$ is the inverse image of
$\Gal(LK'/K') \subseteq \Gal(L/K)$ in $\AST(A)$.
\item
Any decomposition of $\QQ^{2g}$ into indecomposable $\AST(A)$-modules corresponds to a product-up-to-isogeny decomposition of $A$.
\item
The group $\AST(A)$ is a torus if and only if $A$ is isogenous to a product of abelian varieties with CM defined over $K$.
\end{itemize}
\end{defn}

\begin{example}
Put
\[
M_1 := \QQ(i), M_2 := \QQ(\sqrt{-2}), M_3 := \QQ(\sqrt{-3}),
M_4 := \QQ(\sqrt{-6})
\]
and let $K$ be the compositum of these four fields. Let $A$ be the product of four elliptic curves $E_1,\dots,E_4$ with CM by $M_1,\dots,M_4$, respectively. Then $\AST(A)$ is a torus of dimension 3.
\end{example}

\begin{remark} \label{R:ST vs AST}
The \emph{Sato-Tate group} of $A$, as studied for abelian surfaces in \cite{fkrs},
is a maximal compact subgroup of $\AST(A, \CC)$; it therefore has the same component group as $\AST(A)$.
On one hand, this means that the argument using algebraic Sato-Tate groups in the proof of Theorem~\ref{T:Sato-Tate group} directly applies also to the component groups of Sato-Tate groups; on the other hand, the conclusion of Theorem~\ref{T:Sato-Tate group} in the case $g=2$ is a consequence of \cite[Theorem~4.3]{fkrs}, which will save a bit of case analysis.
\end{remark}

We prove the upper bound assertion of Theorem~\ref{T:endomorphism field} by establishing the following result.
\begin{theorem} \label{T:Sato-Tate group}
For $A$ an abelian variety of dimension $g$ over a number field $K$, the component group of the algebraic Sato-Tate group of $A$ (or equivalently, the Sato-Tate group of $A$)
has order dividing $\prod_p p^{r'(g,p)}$.
\end{theorem}
\begin{proof}
Put $G = \AST(A)$. It suffices to check the claimed divisibility for the $p$-part of $\#\pi_0(G)$; this is immediate from Remark~\ref{R:Minkowski Q} (applied with $n = 2g$) unless $p=2$ or $p$ is a Fermat prime.
In light of Remark~\ref{R:ST vs AST}, we may assume further that $g \geq 3$.

Note that for any fixed $p$, $r'(g,p)$ is superadditive in $g$;
we may thus reduce to the case where $A$ is indecomposable, which as noted above implies that $G$ acts indecomposably on $V = \QQ^{2g}$.
Using Corollary~\ref{C:imprimitive3} and Lemma~\ref{L:imprimitive}, we may also deduce the claim in case $G^\circ$ does not act isotypically on $V$ (the exceptional case of Lemma~\ref{L:imprimitive} cannot occur for $g \geq 3$).
We may thus assume hereafter that $G^\circ$ acts isotypically on $V$. 

Let $W$ be 
an irreducible $G^\circ$-representation occurring in $V$, and put
\[
D := \End_{G^\circ}(W), \qquad M := \End_{G^\circ}(V), \qquad  F := Z(D) = Z(M).
\]
By Schur's lemma, $D$ is a division algebra, $M$ is a matrix ring over $D$,
and $T := \image(G^\circ \to M^\times)$ is a torus which splits over $F$.
If we define $H := \ker(G \to \Out(G^\circ))$, we obtain an induced injective morphism $H/G^\circ \hookrightarrow M^\times/T$.
Meanwhile, since $V$ is $G^\circ$-isotypical,
$G/H$ acts faithfully on the set of isomorphism classes of 
$G^\circ$-constituents of $W \otimes_{\QQ} \overline{\QQ}$;
this implies that
the map $G \to \Aut(F/\QQ)$ induces an injective morphism $G/H \hookrightarrow \Aut(F/\QQ)$.

Define the following positive integers:
\begin{align*}
a &:= [F:\QQ]; \\
b &:= \rank_F D = \mbox{the $G^\circ$-multiplicity of $W \otimes_F \overline{F}$}; \\
c &:= \rank_D M = \mbox{the $G^\circ$-multiplicity of $W$ in $V$}; \\
d &:= \frac{2g}{abc} = \mbox{the $\overline{\QQ}$-dimension of a $G^\circ$-constituent of $V \otimes_\QQ \overline{\QQ}$.}
\end{align*}
In this notation, $G/H$ injects into $S_a$ while $H/G^\circ$ injects into a subgroup of a twisted form of $\GL(bc)_F$. In light of Remark~\ref{R:twisted forms}, it follows that the $p$-adic valuation of $\#\pi_0(G)$ is at most
\begin{equation} \label{eq:decomp bound}
r(\Aut(F/\QQ), p) + r(\GL(bc)_F, p) \leq r(\GL(abc)_\QQ, p).
\end{equation}
If $d  > 1$, then Remark~\ref{R:double} immediately yields the desired bound (even for $p=2$).
We may thus assume hereafter that $d=1$, which implies that $G^\circ$ is abelian and hence a torus. In this case, $F$ must be totally imaginary; it is in fact the CM field associated to the unique simple isogeny factor of $A$. 

If $p>2$ is a Fermat prime, then the only way to violate the desired bound would be to have equality in \eqref{eq:decomp bound}, which can only occur if $F \subseteq \QQ(\mu_p)$ 
in light of Corollary~\ref{C:imprimitive2} (the degree-$p$ subextension of $\QQ(\mu_{p^2})$ cannot be a compositum of CM fields because its degree is odd).
Since $F$ is totally imaginary, this would force $F = \QQ(\mu_p)$. However, under these conditions, we may invoke Remark~\ref{R:sharp Minkowski}: the reduction of $G^\circ$ itself always has order divisible by $p$, yielding exactly the correct bound.

If $p=2$ and $F \cap \QQ(\mu_{2^\infty}) = \QQ$, then $r(\GL(n)_F, 2) = r(\GL(n)_\QQ, 2)$
(see Remark~\ref{R:Minkowski K}) and so we may invoke Lemma~\ref{L:imprimitive}
to deduce the desired result (again assuming that $g \geq 3$).
If instead $\QQ(\mu_4) \subseteq F$, then Remark~\ref{R:base extension} gives a valuation bound which is only off by 2, and again this discrepancy is accounted for by Remark~\ref{R:sharp Minkowski} (the reduction of $G^\circ$ itself always has order divisible by $4$).
Otherwise, $F' := F \cap \QQ(\mu_8)$ must equal one of $\QQ(\sqrt{2})$ or $\QQ(\sqrt{-2})$.
In case $F' = \QQ(\sqrt{-2})$, Remark~\ref{R:Minkowski K} implies that $r(\GL(n)_{F'}, 2) = r(\GL(n)_\QQ, 2)$; we thus  obtain an upper bound of 
\[
r(\Aut(F'/\QQ), 2) + r(\GL(abc/2)_{F'}, 2) =
1 + r(\GL(g)_{\QQ}, 2)
\]
and again Lemma~\ref{L:imprimitive} settles the question (for $g \geq 3$).

In case $F' = \QQ(\sqrt{2})$, we must argue a bit more carefully. Since $F'$ is Galois and totally real, we must have $F' \neq F$ and $r(\Aut(F/\QQ),2) = 1 + r(\Aut(F/F'), 2)$. By Remark~\ref{R:Minkowski K}, we have
\[
r(\GL(n)_F,2) = r(\GL(n)_{F'}, 2) = r(\GL(n)_\QQ, 2) + \left\lfloor \frac{n}{2} \right\rfloor.
\]
Note finally that the reduction of $G^\circ$ always has order divisible by 2. 
From \eqref{eq:decomp bound}, we now obtain an upper bound of
\[
r(\Aut(F/F'), 2) + r(\GL(bc)_\QQ, 2) + \left\lfloor\frac{bc}{2} \right\rfloor
\]
which by Lemma~\ref{L:imprimitive} (and the fact that $a \geq 4$) is itself bounded above by
\[
r(\GL(g)_{\QQ},2) + \left\lfloor\frac{g}{a} \right\rfloor \leq
r(\GL(g)_{\QQ},2) + \left\lfloor\frac{g}{4} \right\rfloor.
\]
This gives the desired bound once more.
\end{proof}

\section{Lower bounds}

To conclude, we establish the lower bound assertion of Theorem~\ref{T:endomorphism field}
by twisting powers of CM abelian varieties.

\begin{defn}
We briefly recall \cite[Definition~2.20]{fkrs}. Let $A$ be an abelian variety over a number field $K$, let $L/K$ be a finite Galois extension, and let $f: \Gal(L/K) \to \Aut(A_L)$ be a 1-cocycle.
Then there exists an abelian variety $A^f$ over $K$ equipped with an isomorphism $A^f_L \cong A_L$ such that the action of $\tau \in G_K$
on $A^f(\overline{K}) \cong A^f_L(\overline{K})$ corresponds to the action of $f(\tau) \tau$
on $A(\overline{K}) \cong A_L(\overline{K})$.
The isomorphism $A^f_L \cong A_L$ induces an isomorphism $\End(A^f_L) \cong \End(A_L)$ in which corresponding elements $\alpha \in \End(A^f_L), \beta \in \End(A_L)$ satisfy
the relation
\begin{equation} \label{eq:twist relation}
\tau(\alpha) = f(\tau)\tau(\beta) f(\tau)^{-1}.
\end{equation}
\end{defn}

We use the twisting setup in the following setting. 
\begin{defn} \label{D:twisting setup}
Fix a prime $p$ and a positive integer $m$.
Let $A_0$ be an abelian variety of dimension $g_0$ over some number field $K$,
such that $A_0$ has complex multiplication by the ring of integers $\frako_M$ of a subfield $M$ of $\QQ(\mu_{p^m})$; put $d = [\QQ(\mu_{p^m}):M]$.
(Note that we cannot hope to fix the field $K$, because its degree over $\QQ$ is related to the class number of $M$.)
Let $G_0 \subseteq \GL(d, M)$ be a subgroup of order $p^m$ stable under $\Gal(M/\QQ)$
and identify $G_0$ with a subgroup of $\Aut(A^d_{0,KM})$.
Put $A_1 = A_0^{dn}$ for some positive integer $n$;
then $G_1 = G_0 \wr S_n$ may be identified with a subgroup of $\Aut(A_{1,KM})$ stable under $\Gal(KM/K)$.
Let $G$ be the image of $G_1$ under the map $\GL(dn, M) \to \PGL(dn, M)$.

Choose an $S_n$-extension $L_0/K$ linearly disjoint from $KM$, so that $L_0 M/KM$ is again an $S_n$-extension.
Note that for a ``generic''  $C_{p^m}$-extension $L_1/L_0 M$, the Galois closure $L_2$ of $L_1$ over $M$ will have Galois group $G_0 \wr S_n$.
Using class field theory, we may further ensure that $L_2$ is also Galois over $K$ and that there exists a 1-cocycle $f: \Gal(L_2/K) \to \Aut(A_{1,L_2})$
whose restriction to $\Gal(L_2/KM)$ is the preceding identification of the latter with
$G_0 \wr S_n \subseteq \Aut(A_{1,KM})$.

Put $A = A_1^f$ and let $L$ be the endomorphism field of $A$. Then $KM \subseteq L$
and \eqref{eq:twist relation} implies the existence of a surjective morphism from $\Gal(L/KM)$ to $G$, but not in general to $G_1$.
\end{defn}

\begin{theorem}
For each prime $p$, there exists an abelian variety $A$ of dimension $g$ over some number field $K$ such that the $p$-part of $[L:K]$ is at least $p^{r'_{g,p}}$.
\end{theorem}
\begin{proof}
Suppose first that $p-1$ is not a power of 2 (i.e., $p$ is odd and not a Fermat prime).
Then there exists a subfield $M$ of $\QQ(\mu_p)$ whose index $\ell$ is an odd prime divisor of $p-1$. 
Applying Definition~\ref{D:twisting setup} with $m=1$, $n = \left\lfloor \frac{g}{p-1} \right\rfloor$ then yields the desired result; note that in this case, there is no harm to take $K$ to contain $M$, which simplifies the analysis somewhat.

Suppose next that $p$ is a Fermat prime; we may assume that $r_{g,p} \geq 1$. The previous construction breaks down because $p-1$ admits no odd prime divisor,
so we can only apply Definition~\ref{D:twisting setup} with $m=1$, $M = \QQ(\mu_p)$, $n = \left\lfloor \frac{g}{p-1} \right\rfloor$. Note that we now lose one factor of $p$ to the quotient map $G_1 \to G$, so again we get the desired result.

Suppose finally that $p=2$. Apply Definition~\ref{D:twisting setup} with $m=2$, $M = \QQ(i)$, $n = g$; note that in this case we may even take $K = \QQ$. This time, we lose two factors of 2 to the quotient map $G_1 \to G$, but gain one back from the extension $M/\QQ$. This proves the claim once more.
\end{proof}

\end{document}